\pdfoutput=1
\RequirePackage{ifpdf}
\ifpdf 
\documentclass[pdftex]{sigma}
\else
\documentclass{sigma}
\fi

\usepackage{mathtools}
\usepackage{tikz}
\usepackage{tikz-cd}
\usetikzlibrary{positioning, automata, patterns, decorations.pathreplacing, arrows, shapes, calc}

\usepackage[mode=buildnew]{standalone}

\usepackage{mymacros}

\newtheorem{Theorem}{Theorem}[section]
\newtheorem*{Theorem*}{Theorem}
\newtheorem{Corollary}[Theorem]{Corollary}
\newtheorem{Conjecture}[Theorem]{Conjecture}
\newtheorem{Lemma}[Theorem]{Lemma}
\newtheorem{Proposition}[Theorem]{Proposition}
 { \theoremstyle{definition}
\newtheorem{Definition}[Theorem]{Definition}

\newtheorem{Example}[Theorem]{Example}
\newtheorem{Remark}[Theorem]{Remark}

}

\makeatletter
\g@addto@macro\bfseries{\boldmath}
\makeatother

\begin{document}
\allowdisplaybreaks

\newcommand{\arXivNumber}{2412.02042}

\renewcommand{\PaperNumber}{091}

\FirstPageHeading

\ShortArticleName{$\Delta$ Invariants of Plumbed Manifolds}

\ArticleName{$\boldsymbol{\Delta}$ Invariants of Plumbed Manifolds}

\Author{Shimal HARICHURN~$^{\rm a}$, Andr\'as N\'EMETHI~$^{\rm bcde}$ and Josef SVOBODA~$^{\rm f}$}

\AuthorNameForHeading{S~Harichurn, A.~N\'emethi and J.~Svoboda}

\Address{$^{\rm a)}$~School of Mathematics, Statistics and Computer Science, University of KwaZulu-Natal,\\
\hphantom{$^{\rm a)}$}~South Africa}
\EmailD{\href{mailto:sharichurn.research@gmail.com}{sharichurn.research@gmail.com}}

\Address{$^{\rm b)}$~Alfr\'ed R\'enyi Institute of Mathematics, Re\'altanoda utca 13-15, 1053 Budapest, Hungary}
\EmailD{\href{mailto:nemethi.andras@renyi.hu}{nemethi.andras@renyi.hu}}

\Address{$^{\rm c)}$~Department of Mathematics, University of Budapest (ELTE),\\
\hphantom{$^{\rm c)}$}~P\'azm\'any P\'eter S\'et\'any~1/A, 1117, Budapest, Hungary} \Address{$^{\rm d)}$~Babe\c{s}-Bolyai University, Str. M.~Kog\u{a}lniceanu~1, 400084 Cluj-Napoca, Romania}
\Address{$^{\rm e)}$~Basque Center for Applied Mathematics (BCAM), Alameda de Mazarredo~14,\\
\hphantom{$^{\rm e)}$}~48009 Bilbao, Spain}

\Address{$^{\rm f)}$~Department of Mathematics, California Institute of Technology, Pasadena, CA 91125, USA}
\EmailD{\href{mailto:svo@caltech.edu}{svo@caltech.edu}}

\ArticleDates{Received February 15, 2025, in final form October 17, 2025; Published online October 24, 2025}

\Abstract{We study the minimal $q$-exponent $\Delta$ in the BPS $q$-series $\widehat{Z}$ of negative definite plumbed 3-manifolds equipped with a spin$^{\rm c}$-structure. We express $\Delta$ of Seifert manifolds in terms of an invariant commonly used in singularity theory. We provide several examples illustrating the interesting behaviour of $\Delta$ for non-Seifert manifolds. Finally, we compare $\Delta$ invariants with correction terms in Heegaard--Floer homology.}


\Classification{57K30; 57K31; 57K16}

\section{Introduction}

$\Zhat_b(Y, q)$ is a $q$-series invariant of a negative definite plumbed rational homology 3-spheres equipped with a $\spc$ structure $b$~\cite{GPPV20}. It recovers Witten--Reshetikhin--Turaev $U_q(\mathfrak{sl}_2)$-invariants in radial limits to roots of unity as conjectured by~\cite{GPPV20}, and recently proved in~\cite{Mur23}.

In this paper, we focus on the behavior of \smash{$\Zhat_b(Y, q)$} near $q=0$. In other words, we study the smallest $q$-exponents $\Delta_b$ in \smash{$\Zhat_b(Y,q)$},
\[
 \Zhat_b(Y,q) = q^{\Delta_b}\bigl(c_0+c_1 q + c_2 q^2 + \cdots\bigr), \qquad c_0 \neq 0.
\]
The rational numbers $\Delta_b$ were studied in~\cite{GPP21} where their fractional part was related to various invariants of 3-manifolds. In this work, we focus on the actual value of $\Delta_b$.

For plumbed manifolds, it is natural to consider the `canonical' $\spc$ structure $\can$. Related to it, there is a numerical topological invariant $\gamma(Y) := k^2+s$ (see Section~\ref{ss:spinc}).
It appears, e.g., in the study of Seiberg--Witten invariants~\cite{NeNi02} of the associated plumbed 3-manifold, and its use in topology goes back to Gompf~\cite{Gompf98}. It also plays an important role in singularity theory, e.g., in Laufer's formula~\cite{Laufer1977} for the Milnor number of a Gorenstein normal surface singularity.

As the main result of this paper, we prove that for Seifert manifolds, $\Delta_{\can}(Y)$ can be expressed using $\gamma(Y)$.

\begin{Theorem}\label{thm:seifert_hom}
 Let $Y = M(b_0;(a_1, \omega_1), \dots, (a_n, \omega_n))$ be a Seifert manifold associated with a~negative definite plumbing graph. Let $\can$ be the canonical $\spc$ structure of $Y$. Then $\Delta_{\can}$ satisfies
\begin{gather*}
 \Delta_{\can} = -\frac{\gamma(Y)}{4} + \frac{1}{2}.
\end{gather*}
If $Y$ is not a lens space, then $\Delta_{\can}$ is minimal among all $\Delta_b$, $b \in \spc(Y)$.
\end{Theorem}

The idea of the proof is to express $\Delta_{\can}$ as a minimum of a quadratic form, given by the symmetric intersection form of the plumbing graph, over certain integral vectors. The key is identifying this set over which we minimize (Lemma~\ref{lm:delta_minimizer}). We show that in this set, there exists a unique vector (up to sign) that minimizes the quadratic form.

Once we leave Seifert manifolds, the computation of $\Delta_b$ invariants becomes much more complicated. We illustrate this on plumbing graphs with exactly two vertices of degree 3 and no vertices of degree $\geq 4$ in Section~\ref{s:H_shaped}. In this case, $\Delta_{\can}$ is often smaller than \smash{$-\frac{1}{4}\gamma(Y) + \frac{1}{2}$}, because we have a larger freedom in finding minimizing vectors. To analyze $\Delta$ of these graphs, we use splice diagrams~\cite{NeumannWallBook86,NW05a, Sie80} of plumbed manifolds, building on~\cite{GKS23}.

Surprisingly, $\Delta_{\can}$ can also be larger than \smash{$-\frac{1}{4}\gamma(Y) + \frac{1}{2}$}, as a result of interesting cancellations in the formula for $\Zhat_b(Y)$, see Example~\ref{ex:cancel_H}. Namely, $\Zhat_b(Y)$ can be expressed as a weighted~sum over certain lattice points and those can sometimes be organized in pairs with weights of opposite signs.
The cancellation often can be avoided by refining the weights, as provided by~the two-variable series \smash{$\Zhathat_b(q,t)$} defined in~\cite{AJK23}. However, Example~\ref{ex:t_cancel_0} shows that cancellations may occur even in \smash{$\Zhathat_b(q,t)$}.

An analogy between $\Delta_b(Y)$ and the correction terms $d_b(Y)$ in Heegaard--Floer homology was proposed in~\cite{GPP21}. In~\cite{harichurn2023deltaa}, the first author demonstrated on Brieskorn spheres that unlike~$d_b(Y)$,~$\Delta_b(Y)$ are not cobordism invariants.

Using the explicit formula for $\Delta_{\can}$ of Seifert manifolds given by Theorem~\ref{thm:seifert_hom}, we compare~$\Delta_{\can}$ and~$d_{\can}$ for some classes of Brieskorn spheres, where $d_{\can}$ is explicitly known~\cite{BorNem2011}. We find that $\Delta_{\can}$ is generically much larger than $d_{\can}$. The reason for this discrepancy is that although~$\Delta_{\can}$ and~$d_{\can}$ are both minima of certain closely related quadratic forms, for $\Delta_{\can}$, the quadratic form is being minimized over a much smaller set of vectors than the one for $d_{\can}$.

The invariants $\Zhat_b(Y,q)$ (and hence $\Delta_b(Y)$) were defined for a class of plumbed 3-manifolds given by \emph{weakly negative definite} plumbing graphs~\cite[Definition~4.3]{GM21}. In the appendix, we prove that each such plumbing graph can be transformed to a (strictly) negative definite plumbing graph by a sequence of Neumann moves. This means that the classes of 3-manifolds given by weakly negative definite graphs and by negative definite graphs coincide. For this reason, we do not use the weak notion in this paper.

\section{Preliminaries}

In this section, we collect the necessary material on plumbed manifolds and $\spc$ structures on them.

\subsection{Plumbed 3-manifolds}\label{ss:plumbing}

Let $\Gamma = (V,E,m)$ be a finite tree with the set of vertices $V$ and edges $E$, and a vector $m \in \ZZ^{\lvert V \rvert}$ consisting of integer decorations $m_v$ for each vertex $v \in V$. Let $s = \lvert V\rvert$ and let $\delta = (\delta_v)_{v \in V} \in \ZZ^s$ be the vector of the degrees (valencies) of the vertices. We often implicitly order vertices of $V$, so that $V = \{v_1,v_2,\dots,v_s\}$ and write the quantities associated to $v_i$ with subscript $i$.
We can record $\Gamma$ using an $s\times s$ \emph{plumbing matrix} $M=M(\Gamma)$, defined by
\[
 M_{vw} =
 \begin{cases}
 m_v & \text{if} \ v = w, \\
 1 & \text{if} \ (v, w) \in E, \\
 0 & \text{otherwise.}
 \end{cases}
\]
We always assume that $M$ is a negative definite matrix. For a vector $l \in \Z^s$, we write
\[
 l^2 = l^T M^{-1}l.
\]
Note that the quadratic form $l \mapsto -l^2$ is positive definite.

From $\Gamma$, we can construct a closed oriented 3-manifold $Y := Y(\Gamma)$ by \emph{plumbing}~\cite{Neu81}. For each vertex $v$, we consider a circle bundle over $S^2$ with Euler number $m_v$. Then we glue the bundles together along tori corresponding to the edges in $E$. Manifolds given by this construction, with~$M$ negative definite, are called \emph{negative definite plumbed manifolds}.

From the construction above, it follows that $Y$ is a \emph{rational homology sphere}, that is, $H_1(Y) = H_1(Y,\ZZ)$ is finite. We have $H_1(Y) \cong \ZZ^s/M \ZZ^s$ and if we let $H = H_1(Y)$ and let $\lvert\cdot\rvert$ denote cardinality, then we have $\lvert H\rvert=\lvert\det M\rvert$. The quadratic form $l \mapsto -l^2$ takes values in $\lvert H\rvert^{-1}\ZZ$ as the adjugate matrix $\adj M = (\det M) M^{-1}$ of $M$ has integer entries.

\subsection[Spin\^{}c structures]{$\boldsymbol{\Spc}$ structures}\label{ss:spinc}

The set $\spc(Y)$ of $\spc$ structures on $Y$ admits a natural free and transitive action of $H_1(Y)$, hence it is finite. It can be identified with the set $(2\ZZ^s + m)/2M\ZZ^s$ of \emph{characteristic vectors}. For us, it is convenient to use another identification
\begin{equation}\label{eq:spinc_identif}
\spc(Y) \cong (2\ZZ^s + \delta)/2M\ZZ^s.
\end{equation}
This identification is obtained from the usual characteristic vectors via the map $l \mapsto l - Mu$, where $u = (1, 1, \dots, 1)$, which is justified by the identity $ \delta + m = Mu$.

For a vector $b \in 2\ZZ^s + \delta,$ we denote the corresponding $\spc$ structure as $[b] \in (2\ZZ^s + \delta)/2M\ZZ^s$. We often omit the brackets when it is clear from the context, e.g., we write \smash{$\Zhat_b$} for \smash{$\Zhat_{[b]}$}.

We consider the vector $2u - \delta$ and the corresponding `canonical' $\spc$ structure $\can = [2u-\delta]$. Its characteristic vector is $k = 2u-\delta + Mu = m + 2u$. The rational number
\[
\gamma(Y) := k^2+s = (2u-\delta +Mu)^2+s
\]
does not depend on the plumbing representation of $Y$, so it is a topological invariant of $Y$.
Denote by $\Tr(M) = \sum_{v \in V}m_v$ the trace of the plumbing matrix $M$. Then $\gamma(Y)$ can be expressed as follows:
\begin{Proposition}[\cite{NeNi02}]\label{base-k2-s}
 Let $Y:=Y(\Gamma)$ be a negative definite plumbed manifold, which is a~rational homology sphere. Then
 \begin{align*}
 \gamma(Y) &{}= 3s + \Tr(M) + 2 + \sum_{v,w \in V} (2-\delta_v)(2-\delta_w)M_{vw}^{-1}\\
 &{}=3s + \Tr(M) + 2 + (2u-\delta)^2.
 \end{align*}
\end{Proposition}

\subsection[The Zhat\_b invariants]{The $\boldsymbol{\Zhat_b}$ invariants}

For the rest of the paper, let $Y = Y(\Gamma)$ be a negative definite plumbed rational homology sphere with a chosen negative definite plumbing matrix $M$ of $\Gamma$.

\begin{Definition}
Let $b \in 2\ZZ^s + \delta$ be a vector representing a $\spc$ structure $[b]$. For $\lvert q \rvert<1$, the GPPV $q$-series $\Zhat_b(Y,q)$ is defined by
\begin{equation}\label{def:zhat}
 \Zhat_b(Y,q) = q^{\frac{-3s - \Tr(M)}{4}} \vp\oint\limits_{\lvert z_i \rvert = 1}\prod_{v_i\in V} \frac{{\rm d}z_i}{2\pi {\rm i} z_i} \biggl(z_i - \frac{1}{z_i}\biggr)^{2-\delta_i} \Theta_{b}(z),
\end{equation}
where
\begin{equation}\label{eq:theta}
 \Theta_{b}(z) = \sum_{l \in 2M\ZZ^s + b} z^l q^{-\frac{l^2}{4}},
\end{equation}
and $\vp$ denotes the principal part of the integral and \smash{$z=\prod_{i=1}^s z_i^{l_i}$}.
\end{Definition}

We often omit the manifold $Y$ when it is clear from the context, e.g., we write~$\Zhat_b(q)$ for $\Zhat_b(Y,q)$. Clearly, $\Zhat_b(q)$ does not depend on the choice of $b \in [b]$. The convergence of $\Zhat_b(Y)$ for $\lvert q \rvert<1$ follows from the negative definiteness of $M$.

\section[Delta\_b invariants]{$\boldsymbol{\Delta_b}$ invariants}

\begin{Definition}
The \emph{Delta invariant} $\Delta_b$ is the smallest $q$-exponent in $\Zhat_b(q)$. If $\Zhat_b(q) = 0$, we set $\Delta_b = \infty$.
\end{Definition}

Using $\Delta_b$, we can write the $\Zhat_b(q)$ in the following form:
\[
 \Zhat_b(q) = 2^{-s} q^{\Delta_b}\sum_{i=0}^\infty r_iq^i = 2^{-s} q^{\Delta_b}\bigl(r_0 + r_1q^1 + r_2q^2 + \cdots\bigr)
\]
for some integers $r_i$ with $r_0 \neq 0$. The integrality of $r_i$ can be seen directly from the defining formula \eqref{def:zhat} of $\Zhat_b(Y,q)$. The form of the exponents is justified by part (1) of the following Lemma (recall that $\lvert H\rvert$ is the order of $H_1(Y,\ZZ)$).

\begin{Lemma}
 \leavevmode
 \begin{enumerate}\itemsep=0pt
 \item[$(1)$] The differences between the exponents of $\Zhat_b(q)$ are integers.
 \item[$(2)$] The fractional part of the exponents $($and in particular of $\Delta_b)$ is given by
 \[
 \frac{-3s - \Tr(M) -b^2}{4} {\pmod{1}}.
 \]
 Consequently $4 \lvert H \rvert \Delta_b \in \ZZ \cup \{\infty\}$.
 \item[$(3)$] $\lvert H\rvert\Delta_a \equiv \lvert H\rvert\Delta_{b} \pmod{1}$ for any two $\spc$ structures $[a]$, $[b]$ on $Y$.
\end{enumerate}
\end{Lemma}

\begin{proof}
$(1)$ Consider two vectors $l, l+2Mn \in [b]$, where $l,n \in \ZZ^s$. The difference of the corresponding exponents is
 \[
 \frac{-l^2}{4}+\frac{(l+2Mn)^2}{4} = l^Tn + n^T Mn \in \ZZ.
 \]

$(2)$ The exponent of a representative $b \in [b]$ reads
 \[
 \frac{-3s - \Tr(M) - b^2}{4} \in \frac{1}{4\lvert H\rvert}\ZZ
 \]
 By (1), all the other exponents have the same fractional part.

$(3)$
 We have
 \[
 \lvert H\rvert\Delta_b(Y) - \lvert H\rvert\Delta_{a}(Y) \equiv \frac{\lvert H\rvert\bigl(b^2 -a^2\bigr)}{4} \pmod{1}.
 \]
 By writing $b = 2l + \delta$ and $a = 2n+\delta$ with $l,n \in \ZZ$, we see that
 \[
 \frac{1}{4}\lvert H\rvert \bigl(b^2 - a^2\bigr) = \lvert H\rvert \bigl(l^2 + l^TM^{-1}\delta - n^2 + n^TM^{-1}\delta\bigr) \in \ZZ.
\tag*{\qed}
\]
\renewcommand{\qed}{}
\end{proof}

\begin{Remark}
The set of $\spc$ structures admits a natural involution, called conjugation, given by \smash{$[b] \mapsto \overline{[b]} := [-b]$}, i.e., the reflection through the origin under the identification~\eqref{eq:spinc_identif}. It is known that \smash{$\Zhat_{[b]\vphantom{\overline{[b]}}}(q) = \Zhat_{\overline{[b]}}(q)$}, hence we have \smash{$\Delta_{[b]\vphantom{\overline{[b]}}} =\Delta_{\overline{[b]}}$}.
\end{Remark}

\subsection[The exponents of Zhat\_b]{The exponents of $\boldsymbol{\Zhat_b}$}

The $\Delta_b$ invariant is the minimal $q$-exponent in the series $\Zhat_b(q)$. The exponents are, up to an overall shift, given by the quadratic form $l \mapsto -l^2$. Here $l$ are lattice vectors that run over a~certain subset $\LC_b \subset \ZZ^s$. We need to identify this subset.

We order the set of vertices $V$ of the plumbing tree $\Gamma$ by their degree,
\[
 V = \{v_1, \dots, v_{s_1}, v_{s_{1}+1}, \dots, v_{s_{1}+s_{2}}, v_{s_{1}+s_{2}+1}, \dots, v_{s_{1}+s_{2} + s_{3} = s}\}.
\]
We have $s_1$ leaves, $s_2$ vertices of degree 2 and $s_3$ vertices of degree $\geq 3$.

The integrand of $\Zhat$ contains the rational function
\begin{equation}\label{eq:rat_function}
 \prod_{i=1}^s\bigl(z_i- z_i^{-1}\bigr)^{2-\delta_i} = \prod_{i=1}^{s_1} \bigl(z_i-z_i^{-1}\bigr) \prod_{i=s_1+s_2+1}^{s} \frac{1}{\bigl(z_i-z_i^{-1}\bigr)^{\delta_i-2}}.
\end{equation}
The integration in \eqref{def:zhat} is equivalent to the following procedure: First, we expand each term of the product above using the \emph{symmetric expansion}~-- the average of Laurent expansions as $z_i \rightarrow 0$ and $z_i \rightarrow \infty$, and multiply these together, giving an element of $\Z\bigl[z_1^{\pm 1},\dots, z_s^{\pm 1}\bigr]$,
\begin{equation}\label{eq:rat_function_expn}
 \prod_{i=1}^s \se \bigl(z_i- z_i^{-1}\bigr)^{2-\delta_i} = \sum_{l \in \ZZ^s} \ct_l z^l = \sum_{l \in \ZZ^s} \ct_l z_1^{l_1}z_2^{l_2}\cdots z_s^{l_s}.
\end{equation}
Then we multiply the result with the theta function $\Theta_b(z)$ in \eqref{eq:theta}, and we extract the constant coefficient in variables $z_i$, giving the $q$-series $\Zhat_b(q)$.

Let $\LCt$ denote the set of all the vectors $l \in \ZZ^s$ with nonvanishing coefficient $\ct_l \neq 0$ in \eqref{eq:rat_function_expn}. Similarly, define $\LCt_b = \LCt \cap -(2\ZZ^s+b)$. The reason for the sign is that we are pairing $l \in \LCt$ with $-l \in 2\ZZ^s+b$ when extracting the constant coefficient. Note that while $\LCt$ is symmetric about the origin, $2\ZZ^s+b$ in general is not.

\begin{Lemma}\label{description-c}
The set $\LCt$ consists of vectors whose components satisfy the following conditions:
 \[\LCt =
 \left\{l=(l_1,\dots,l_{s_1}, 0, \dots 0, t_{s_1+s_2+1}, \dots, t_{s_3}) \in \ZZ^s
 \;\middle\vert\;
 \begin{aligned}
 &l_i = \pm 1 \\
 &t_i \equiv \delta_i \pmod{2} \\
 &\lvert t_i \rvert \geq \delta_i-2
 \end{aligned}
 \right\}.
\]
\end{Lemma}

\begin{proof}
 The expansion for leaves is simply $z_i-z_i^{-1}$, giving the entries $l_i = \pm 1$. The variables corresponding to vertices of degree 2 are absent in \eqref{eq:rat_function_expn}. For a vertex $v_i$ of degree $\delta_{i} \geq 3$, put $d := \delta_{i}-2 \geq 1$. We have the following symmetric expansion:
 \begin{align*}
 2 \cdot \se \bigl(z-z^{-1}\bigr)^{-d} &{}=
 \underset{z \to \infty}{\expn} \frac{z^{-d}}{\bigl(1-z^{-2}\bigr)^{d}}
 + \underset{z \to 0}{\expn} \frac{z^{d}}{\bigl(z^2-1\bigr)^{d}}\\
 &{}=
 z^{-d} \sum_{k=0}^\infty \binom{k-1+d}{k}z^{-2k} + (-1)^d z^{d} \sum_{k=0}^\infty \binom{k-1+d}{k} z^{2k}\\
 &{}=
 \bigl(z^{-d} + d z^{-d-2} + \cdots\bigr) + (-1)^{d} \bigl(z^{d} + d z^{d+2} +\cdots\bigr).
 \end{align*}
 It follows that the corresponding entry $t_i$ must have the same parity as $\delta_i$ and $\lvert t_i \rvert \geq d = \delta_i-2$.
\end{proof}

\subsection{Cancellations}\label{ss:cancel}

By the previous section, the $q$-series $\Zhat_b(Y,q)$ can then be expressed as a sum over $\LCt_b$,
\begin{equation*}
 \Zhat_b(Y,q) = q^{\frac{-3s - \Tr(M)}{4}} \sum_{l \in \LCt_b} \ct_l q^{\frac{-l^2}{4}}.
\end{equation*}
The $q$-exponents of $\Zhat_b(q)$ are therefore given by
$\bigl( -3s - \Tr(M) -l^2\bigr)/4$ for those $l \in \LCt_b$ whose contribution does not cancel out in the sum above, in other words,
\begin{equation}\label{eq:cancel}
 c_l:= \sum_{\substack{l' \in \LCt_b\\l'^2 = l^2}} \ct_{l'} \neq 0.
\end{equation}
However, $c_l$ can vanish even if there are nonzero terms $\tilde{c}_{l'}$ in \eqref{eq:cancel}. We refer to this phenomenon as `cancellations' -- see Examples \ref{ex:cancel_H} and \ref{ex:cancel_seifert}.
Motivated by this, we define
\[
 \LC_b := \{l \mid l \in \LCt_b;\, l \text{ satisfies \eqref{eq:cancel}} \} \subseteq \LCt_b.
\]
It follows that $\Delta_b$ is determined by minimizing $-l^2$ over the set $\LC_b$. We refer to the elements $l$ in $\LC_b$ for which $-l^2$ is minimal, as \emph{minimizing vectors}.
\begin{Lemma}\label{lm:delta_minimizer}
 \begin{equation*}
 \Delta_b(Y) = \frac14 \Bigl(-3s - \Tr(M) + \min_{l \in \LC_b} \bigl\{ -l^2 \bigr\}\Bigr).
 \end{equation*}
\end{Lemma}

\section{\texorpdfstring{$\Delta$}{Delta} for Seifert manifolds}

In this section, we will prove Theorem~\ref{thm:seifert} which gives an explicit formula for $\Delta_{\can}(Y)$ of a~Seifert manifold $Y$ equipped with the canonical $\spc$ structure $\can$.

\subsection{Seifert manifolds}
 Seifert manifold $Y = M(b_0;(a_1, \omega_1), \dots, (a_n, \omega_n))$, fibered over $S^2$, is given by an integer $b_0$ and tuples of integers $0<\omega_i<a_i$ for $1 \leq i \leq n$, ${\rm gcd}(\omega_i,a_i)=1$. It can be represented by a~star-shaped plumbing as shown in Figure~\ref{fig:seifert}. The graph consists of a central vertex with the decoration $-b_0$ and $n$ `strings'. The decorations \smash{$-b_{j_1},-b_{j_2},\dots, -b_{j_{s_j}}$} on the $j$-th string are determined by the Hirzebruch--Jung continued fraction
 \[
 \frac{a_j}{\omega_j} = \bigl[b_{j_1}, \dots, b_{j_{s_j}}\bigr] =b_{j_1} - \cfrac{1}{b_{j_2} - \cfrac{1}{\ddots - \cfrac{1}{b_{j_{s_j}}}}} .
 \]

\begin{figure}
 \centering
 \includegraphics[scale=1.1]{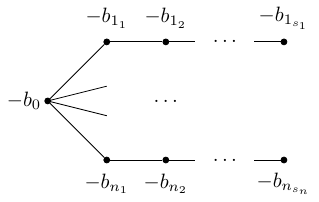}
 \caption{Plumbing graph of a Seifert manifold.}
 \label{fig:seifert}
\end{figure}

The intersection form $M$ is negative definite if and only if the orbifold Euler number
${\rm e} = -b_0 + \sum_i \omega_i/a_i$
is negative (see~\cite[Theorem~5.2.]{Neumann1978}).

We have the following formula for $\gamma(Y)$ of a Seifert manifold $Y$~\cite[p.~296]{NeNi02}, a consequence of Proposition~\ref{base-k2-s}:
 \begin{equation}\label{eq:formula_gamma}
 \gamma(Y) = \frac{1}{{\rm e}}\Biggl(2-n + \sum_{i=1}^n \frac{1}{a_i}\Biggr)^2 + {\rm e} + 5 + 12 \sum_{i=1}^n \ds(\omega_i, a_i).
 \end{equation}
Here $\ds$ denotes the Dedekind sum.

\subsection[Delta\_{can} of Seifert manifolds]{$\boldsymbol{\Delta_{\can}}$ of Seifert manifolds}

\begin{Theorem}\label{thm:seifert}
Let $Y = M(b_0;(a_1, \omega_1), \dots, (a_n, \omega_n))$ be a negative definite Seifert manifold. Then $\Delta_{\can}$ of the canonical $\spc$ structure satisfies
\begin{equation}\label{eq:seifert}
 \Delta_{\can} = -\frac{\gamma(Y)}{4} + \frac{1}{2}.
\end{equation}
If $Y$ is not a lens space, then $\Delta_{\can}$ is minimal among all $\Delta_b$, $b \in \spc(Y)$.
\end{Theorem}

\begin{proof}
We will treat the lens spaces separately in Section~\ref{ssec:lens}, so we may assume that $n \geq 3$ and $a_i \geq 2$ for each $i$. Denote $A = \prod_{i=1}^n a_i$.

We will first show that over the set $\LCt$ from Lemma~\ref{description-c}, $-l^2$ is minimized exactly by the vectors $\pm(2u-\delta)$. From that, it will follow that $2u-\delta \in \LC_{\can}$ is a true minimizing vector in the sense of Section~\ref{ss:cancel}. The formula \eqref{eq:seifert} for $\Delta_{\can}$ then follows from Proposition~\ref{base-k2-s}.

The set $\LCt$ consists of the vectors
\[
 l = (l_1, l_2, \dots, l_n, 0, \dots, 0, t),
\]
where $l_i = \pm 1$ and $t \equiv n \pmod{2}$ and $\lvert t \rvert \geq n-2$ by Lemma~\ref{description-c}. The quadratic form $l^2$ can be expressed in terms of the Seifert data using Neumann's Theorem~\ref{theorem:neumann}
below (see also~\cite[p.~19]{GKS23}):
\begin{equation}\label{eq:quadratic}
 -l^2 = \frac{1}{|H|} \Biggl(t^2 A + 2 \sum_{i=1}^n l_i t\frac{A}{a_i} +
 2 \sum_{i > j}^n l_i l_j \frac{A}{a_ia_j} \Biggr) - \sum_{i=1}^n M^{-1}_{ii}.
\end{equation}

By the symmetry $l^2 = (-l)^2$, we may assume that $t \geq n-2$. Taking the derivative with respect to $t$, we obtain
\[
\frac{1}{A |H|} \partial_t \bigl(-l^2\bigr) = 2t + 2\sum_{i=1}^n \frac{l_i}{a_i} \geq 2t - n,
\]
which implies that $\partial_t\bigl(-l^2\bigr) >0$ for $t \geq n$. This means that \eqref{eq:quadratic} is nondecreasing on the interval $[n,\infty)$. As $t$ and $n$ have the same parity, we need to additionally check that \eqref{eq:quadratic} is smaller for $t = n-2$ than it is for $t = n$:
\[
n^2+2n \sum_{i=1}^n l_i \frac{A}{a_i} \geq (n-2)^2+2(n-2)\sum_{i=1}^n l_i \frac{A}{a_i}.
\]
This follows from
\[n-1\geq \frac{n}{2} \geq -\sum_{i=1}^n l_i \frac{A}{a_i}.
\]

Similarly, pick $j \in 1,\dots, s_1.$ Let $l^+$, $l^-$ be two vectors with $t = n-2$ that only differ in the $j$-th coordinate, say $l_j^+=1$ and $l^-_j=-1$. We have
\begin{align*}
 \bigl(l^+\bigr)^2 - (l^-)^2 & = \frac{2A}{a_j |H|} \biggl(2(n-2) + 2\sum_{i \neq j} \frac{l_i}{a_i}\biggr) \\
 & \geq \frac{2A}{a_j |H|} (2(n-2) - (n-1)) \geq 0.
\end{align*}
It follows that
\[
 l = (-1,\dots, -1, 0, \dots, 0, n-2) = \delta - 2u
\]
minimizes \eqref{eq:quadratic}. The argument also implies that the only other vector giving the same value of~$l^2$ is $-l=2u-\delta$ which represents the canonical $\spc$ structure $\can$. These two vectors have the same coefficients in the expansion of~\eqref{eq:rat_function},
\[
\tilde{c}_l = \tilde{c}_{-l} = \frac{(-1)^n}{2}.
\]
Therefore, even in the case that $\pm l$ belong to the same $\spc$ structure, i.e., $\can$ is $\spin$, their contributions do not cancel in $\Zhat_{\can}(q)$ and consequently $-l \in \LC_{\can}$.
Thus the minimum of $-l^2$ over $\LC_{\can}$ is given by $-(2u-\delta)^2$ and we have
 \begin{equation*}
 \Delta_{\can} = \frac{-3s - \Tr(M)}{4} - \frac{(2u-\delta)^2}{4} = - \frac{\gamma(Y)}{4} + \frac{1}{2}.
 \end{equation*}
 The last equality follows the formula for $\gamma(Y)$ in Proposition~\ref{base-k2-s}.
\end{proof}

\subsection{Lens spaces}\label{ssec:lens}

In this section, we compute $\Delta_b$ of lens spaces. Let $Y= L(p,r) $ be a lens space with $p>r>0$, ${\rm gcd}(p,r)=1$. Denote by $g$ a generator of $H_1(Y,\ZZ) \cong \ZZ/p\ZZ$ (written multiplicatively) and let $\can$ be the canonical $\spc$ structure. All $\spc$ structures are of the form $g^i \can$ for $i=0, \dots, p-1$, i.e., the images of $\can$ under the natural action of $H_1(Y)$ on $\spc(Y)$.
 We have the following formula for \smash{$\Zhat_b(q)$} of lens spaces from~\cite[Section~4.2.]{GKS23}:
 \begin{equation*}
 \sum_{i=0}^{p-1} \Zhat_{g^i \can}(q)g^i
 = q^{3\ds(r,p)} \bigl(\bigl(g^{-r-1}+1\bigr) q^{\frac{1}{2p}} - \bigl(g^{-r}+g^{-1}\bigr) q^{-\frac{1}{2p}}\bigr).
 \end{equation*}
 Here $\ds(r,p)$ denotes the Dedekind sum. From the formula, we read off the four nonzero $\Delta_b$ invariants.
 \begin{equation}\label{eq:delta_lens}
 \Delta_{g^{-r-1}\can} = \Delta_{\can} = 3\ds(r,p) + \frac{1}{2p}, \qquad
 \Delta_{g^{-r}\can} = \Delta_{g^{-1}\can} = 3\ds(r,p) - \frac{1}{2p}.
 \end{equation}
$\gamma(Y)$ can be described using Dedekind sums in the following manner~\cite[p.~304]{NeNi02}:
\[
 \gamma(Y) = 2 - \frac{2}{p} - 12\ds(p,r).
\]
Comparing with \eqref{eq:delta_lens}, we obtain that $\Delta_{\can}$ satisfies the same formula as for the other Seifert manifolds:
\begin{equation*}
 \Delta_{\can} = -\frac{\gamma(Y)}{4} + \frac12.
\end{equation*}

Note that \eqref{eq:delta_lens} shows that for lens spaces, $\Delta_{\can}$ is not minimal among all $\Delta_b$. See also Example~\ref{ex:cancel_H}.

\begin{Remark}
 Theorem~\ref{thm:seifert} and Section~\ref{ssec:lens}
imply that $\Zhat_{\can}(q)$ is always non-zero for negative definite Seifert manifolds fibered over $S^2$. It can still be a single monomial, as is illustrated in Example~\ref{ex:cancel_seifert}.
\end{Remark}

\begin{Remark}
 We originally proved Theorem~\ref{thm:seifert} using the reduction theorem~\cite[Theorem~4.2]{GKS23} which may be used to compute all $q$-exponents of $\Zhat_b(q)$ of Seifert manifolds. Later, we found a simpler argument presented here, which focuses on the smallest exponent. It also emphasizes the role of the vector $2u-\delta$, making the presence of the invariant $\gamma(Y)$ more transparent.
\end{Remark}

\begin{Remark}
 The series $\Zhat_b(q)$ often vanishes (see \cite[Corollary~3.7]{GKS23})
 in which case the $\Delta_b$ invariants are infinite. In~\cite{GPPV20}, the authors conjectured, based on physics considerations, the existence of a categorification of $\Zhat_b(q)$, i.e., a doubly-graded cohomology theory \smash{$\mathcal{H}_b^{i,j}(Y)$} whose graded Euler characteristic is $\Zhat_b(q)$,
 \[
 \Zhat_b(q) = 2^{-s} q^{\Delta_b} \sum_{i, j \in \ZZ} q^i (-1)^j \dim \mathcal{H}_b^{i,j}(Y).
 \]
 Smaller $q$-exponents than $\Delta_b$ could appear in the corresponding two-variable generating series $q^{\Delta_b}\sum_{i,j} q^i t^j \dim \mathcal{H}_b^{i,j}(Y)$. In~\cite{GPV16}, a Poincar\'e series of this sort was defined for lens spaces $L(p,1)$. Its minimal $q$-power is finite for all $\spc$ structures, in contrast with $\Delta_b$.
\end{Remark}

\section{Beyond Seifert manifolds}\label{s:H_shaped}

In the proof of Theorem~\ref{thm:seifert} giving the formula for $\Delta_{\can}$ of Seifert manifolds, we used a special form \eqref{eq:quadratic} of the quadratic from $l \mapsto -l^2$ following from the fact that Seifert manifolds admit a~star-shaped graph. For more general graphs, we need a generalization of \eqref{eq:quadratic}. This is realized by some properties of \emph{splice diagrams}, which are certain weighted graphs built from plumbing graphs. We will illustrate this method on plumbing graphs with exactly two vertices of degree 3 and no vertices of degree 4 or more. The corresponding splice diagrams are `H-shaped' graphs with 6 vertices, as in Figure~\ref{fig:H_splice}.

In general, there is no uniform choice of minimizing vector as was the vector $2u-\delta$ for $\Delta_{\can}$ of Seifert manifolds. Therefore, we cannot hope for a simple universal formula for $\Delta_{\can}$ as in Theorem~\ref{thm:seifert}. Nevertheless, the minimizing vectors keep a specific form in certain regions given by the relative size of the weights in a splice diagram. On the boundaries of these regions, we may see multiple minimizing vectors.

In principle, one can divide the study into those particular cases. Some of these can be effectively reduced to the case of Seifert manifolds, as we will illustrate in Section~\ref{ss:computation_h_shaped}. The techniques described here can be used for more general plumbings, but the number of cases grows significantly with the complexity of the splice diagram.

\subsection{Splice diagrams}

Following~\cite{NW05a} (see also~\cite{SavelievBook}), given a rational homology sphere $Y$ with negative definite
plumbing graph $\Gamma$, we construct the splice diagram $\Omega$ of $Y$ as follows: $\Omega$ is a tree obtained by replacing
each maximal string in $\Gamma$ (a simple path in $\Gamma$ whose interior is open in $\Gamma$) by a single edge. Thus~$\Omega$ is homeomorphic to $\Gamma$ but has no vertices of degree 2. We identify the vertices of $\Omega$ with the corresponding vertices of $\Gamma$.

At each vertex $v$ of $\Omega$ of degree $\geq 3$, we assign a weight $w_{v\varepsilon}$ on an incident edge $\varepsilon$ as follows. The edge $\varepsilon$ in $\Omega$ corresponds to a string in $\Gamma$ starting in $v$ with some edge $e$. Let $\Gamma_{ve}$ be the subgraph of $\Gamma$ cut off by the edge of $\Gamma$ at $v$ in the direction of $e$, as in the following picture:
$$
\includegraphics[scale=1.1]{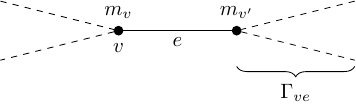}
$$
The corresponding weight $w_{v\varepsilon}$ is given by the determinant of $-M(\Gamma_{ve})$, where $M(\Gamma_{ve})$ denotes the intersection matrix of $\Gamma_{ve}$. We draw the weight $w_{v \varepsilon}$ on the edge $\varepsilon$ near $v$. For example, the following is a plumbing graph and its splice diagram:
$$
\includegraphics[scale=1.1]{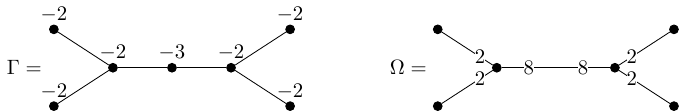}
$$
The splice diagram might depend on the choice of the plumbing graph
(in its class modulo plumbing calculus), however, up to the calculus of the splice diagrams, $\Omega$ is unique (cf.~\cite{Nem22}).
In fact, we can take as a canonical choice the splice diagram associated with the minimal good plumbing graph
(a graph which does not contain any vertices of valency at most 2 with $-1$ decoration).

If $Y(\Gamma)$ is an integer homology sphere, the splice diagram uniquely determines the plumbing graph, see~\cite{NW02}. Although $\Gamma$ and $\Omega$ are essentially equivalent in this case, $\Omega$ is much smaller. More importantly, it is the relative size of the weights of $\Omega$, which directly influences the minimizing vectors. For example, if one of the weights is significantly larger than the others, the minimizing vectors must be of some specific form, see Section~\ref{ss:computation_h_shaped}.

The values of the quadratic form $l \mapsto -l^2=-l^T M^{-1}l$ can be computed from the splice diagram as follows: For two vertices $v$, $v'$ of the splice diagram, consider the shortest path $P$ connecting them. Let $N_{vv'}$ be the product of all weights adjacent to vertices of $P$, but not lying on $P$,\looseness=-1
$$
 \includegraphics[scale=1.1]{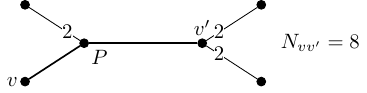}
$$

\begin{Theorem}[{\cite[Theorem~12.2]{NW05a}}]\label{theorem:neumann} With the notation above, we have
 \[
 M^{-1}_{vv'} = -\frac{N_{vv'}}{\det(- M)}.
 \]
\end{Theorem}

Recall that the vectors $l \in \Z^s$ that contribute to $\Zhat_b(q)$ lie in the set $\LCt$ from Lemma~\ref{description-c}. If~${l \in \LCt}$, we have $l_v=0$ if $v$ is of degree 2 and $l_v^2 = (\pm 1)^2=1$ if $v$ is of degree 1. We obtain the following expression for $-l^2$ generalizing \eqref{eq:quadratic}:
\begin{equation*}
 -l^2 = - \biggl(\sum_{\substack{v \neq v'\\ \delta_v,\delta_{v'} \neq 2}} M^{-1}_{vv'} l_v l_{v'} + \sum_{\delta_v \geq 3} M^{-1}_{vv} l_v^2 + \sum_{\delta_v =1} M^{-1}_{vv} \biggr).
\end{equation*}

Note that the coefficients in the first and second sum can be expressed using the weights of the splice diagram $\Omega$, up to the multiplication by $\det(- M)$. On the other hand, the third sum is not expressed in terms of weights of $\Omega$ in a simple way,\footnote{They are the weights at the leaves of the \emph{maximal} splice diagram, see~\cite[p.~743]{NW05a}.} but it is independent of $l$ (whenever $l \in \LCt$), so it does not influence which vectors $l \in \LCt$ have the minimal value of $-l^2$.

\subsection[Delta for homology spheres with H-shaped splice diagrams]{$\boldsymbol{\Delta}$ for homology spheres with H-shaped splice diagrams} \label{ss:computation_h_shaped}

We now consider a plumbing graph $\Gamma$ with exactly two vertices of degree 3 and no vertices of higher degree, e.g., the graph in Figure~\ref{fig:plumbing_from_splice}. For simplicity, we assume that the associated plumbed manifold $Y$ is a homology sphere. The associated splice diagram $\Omega$ is an `H-shaped' graph with six vertices, see Figure~\ref{fig:H_splice}. Its six weights are denoted $a_1$, $a_2$, $a_3$ and $a'_1$, $a'_2$, $a'_3$. They are pairwise coprime integers, which we further assume to be $\geq 2$. In this case, $Y$ can be realized as splicing of Brieskorn spheres $\Sigma(a_1,a_2,a_3)$ and $\Sigma\bigl(a'_1,a'_2,a'_3\bigr)$ along their third singular fibers.

We consider the projection $\ZZ^s \to \ZZ^6$, denoted by $l \mapsto \bar{l}$, which removes the components corresponding to the vertices of degree 2. We order the components of $\bar{l} = \bigl(x_1,x_2,x_3,x'_1,x'_2,x'_3\bigr)$ as in Figure~\ref{fig:H_splice} and keep the ordering throughout this section.

\begin{figure}[t]
 \centering \includegraphics[scale=1.1]{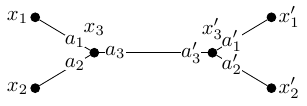}
 \caption{H-shaped splice diagram $\Omega$.}
 \label{fig:H_splice}
\end{figure}

The quadratic form $l \mapsto -l^2$, restricted to the set $\LCt$ from Lemma~\ref{description-c}, can be expressed in terms of the weights of $\Omega$ as
\begin{align}
 -l^2 ={}& x_3^2 a_1 a_2 a_3 +
 2
 x_1 x_3 a_2 a_3 +
 2
 x_2 x_3 a_1 a_3 +
 2
 x_1 x_2 a_3 + (\cdots)' \nonumber\\
 &{}{+}\, 2x_1 x'_2 a_2 a'_1 +
 2x_1' x_3 a_1 a_2 a'_2 +
 2x'_2 x_3 a_1 a_2 a'_1 + (\cdots)' \nonumber \\
 &{}{+}\, 2x_1 x'_1 a_2 a'_2 +
 2 x_2 x'_2 a_1 a'_1 +
 2 x_3x'_3a_1a_2a'_1a'_2 + C.\label{eq:form_H_shaped}
\end{align}

Here $(\cdots)'$ means that we repeat the terms on each line with the usual and dashed variables reversed. $C$ is constant on the set $\LCt$ so it does not influence the minimizing vectors for $\Delta_{\can}$.

The general strategy to identify $\Delta_{\can}$ can be described as follows:\@ We know that $x_1,x_2,x'_1,x'_2 \in \{\pm1\}$ because $l \in \LCt$. For each of the $2^4$ possibilities of the signs, the form above reduces to a~quadratic form in variables $x_3$, and $x_3'$. We can then minimize these forms over odd integers, using standard optimization methods. Finally, we must check that the resulting vectors do not cancel out, as in Section~\ref{ss:cancel}, so they are true minimizing vectors.

Clearly, the form of minimizing vectors depends on the relative size of the coefficients~$a_i$,~$a'_i$. We describe in a greater detail the case when $a_3$ is very large compared to other $a_i$ and $a'_i$. This allows us to effectively reduce the minimizing problem to the Seifert case. If $a_3 \gg a_1,a_2,a'_1,a'_2,a'_3$, the substantial terms are those containing $a_3$
\begin{equation}\label{eq:half_graph}
 x^2_3 a_1 a_2 a_3 + 2x_1 x_3 a_2 a_3 + 2x_2 x_3 a_1 a_3 + 2x_1 x_2 a_3.
\end{equation}

Any vector $l \in \LCt$ with the minimal value of $-l^2$ also minimizes the expression \eqref{eq:half_graph}. As this is (almost) a quadratic form of a star-shaped graph, we can repeat the first part of the argument in the proof of Theorem~\ref{thm:seifert}. Namely, assuming that $a_i \geq 2$ and (without a loss of generality) $x_3 \geq 1$ and taking the $x_3$-derivative,{\samepage
\begin{align*}
\partial_{x_3}( \eqref{eq:half_graph}) = 2 (x_3 a_1 a_2 a_3 + x_1 a_2 a_3 + x_2 a_1 a_3)
 \geq 2 a_3( (a_1-1)(a_2-1) +1)) > 0,
\end{align*}
we obtain that $x_3 = \pm 1$.}

We are left with a quadratic function of a single variable $x'_3$ depending on the choice of signs $x_1,x_2,x_3,x'_1,x'_2 \in \{\pm 1\}$. By taking the $x'_3$-derivative of \eqref{eq:form_H_shaped}, we find that the minimum of $x'_3$ over $\R$ is
\begin{align}
 \min_{\R} x_3' &= -\frac{x'_1 a'_2 a'_3 + x'_2 a'_1 a'_3 + x_1 a'_1 a'_2 a_2 + x_2 a'_1 a'_2 a_1 + x_3 a_1 a_2 a'_1 a'_2}{a'_1 a'_2 a'_3} \nonumber\\
 &=-\biggl(\frac{x_1'}{a'_1}+\frac{x_2'}{a'_2}+\frac{x_1a_2}{a'_3}+\frac{x_2a_1}{a'_3}+\frac{x_3a_1 a_2}{a'_3}\biggr).\label{eq:x_3_prime}
\end{align}
By Lemma~\ref{description-c}, the actual value of $x'_3$ is the odd integer that is the closest to \eqref{eq:x_3_prime}, up to a minor caveat that some cancellations may happen for the resulting vector (see Section~\ref{ss:cancel}).
Taking $a_1=a_2 \gg a'_3$ shows that there are graphs for which the absolute value of the component $x'_3$ in the minimizing vector can be arbitrarily large.\footnote{Again, assuming that there are no cancellations. It should be easy to construct an explicit family of H-shaped graphs where this does not happen.} This is in contrast with the case of Seifert manifolds, where the components of the minimizing vector $2u-\delta$ are numbers bounded by the number of singular fibers.

In a similar way, we could analyze other cases of the relative size of the weights, giving (at least approximate) `explicit formulas' for $\Delta_{\can}$. However, in practice, it is easier to use a~computer to search for the minimizing vectors. One can use a modified version of the code in~\cite{petercode}, as we do in certain cases.
We illustrate the variability of possible minimizing vectors and possible values of $\Delta_{\can}$ on several examples in this and the following section.

\begin{Example}\label{ex:original}
Consider the integral homology sphere $Y$ associated with the splice diagram
$$
 \includegraphics{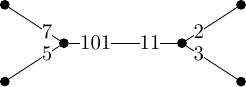}
$$

Then $\Delta(Y) = \frac{3045}{1000}$, with minimizing vectors
$\pm(1,1,-1, -1,-1,3)$. The vector $(1,1,-1, -1,\allowbreak -1,3)$ has $x_3' = 3$ which is consistent with \eqref{eq:x_3_prime} with value \smash{$\frac{193}{66} \doteq 2.92$}.
$\Delta(Y)$ is strictly smaller than \smash{$-\frac{\gamma(Y)}{4} + \frac{1}{2} = \frac{3885}{1000}$} given by the vector $2u-\delta=(-1,-1,1,-1,-1,1)$. The corresponding plumbing graph on $28$ vertices is shown in Figure~\ref{fig:plumbing_from_splice}.

\begin{figure}[ht]
 \centering \includegraphics[scale=1.1]{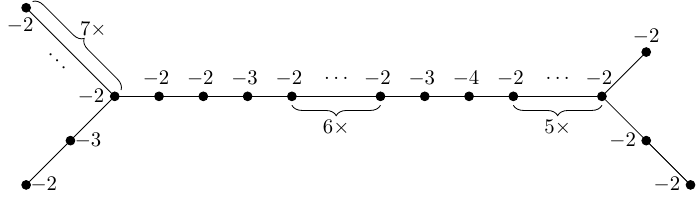}%
 \caption{The plumbing graph associated to the splice diagram in Example~\ref{ex:original}.}
 \label{fig:plumbing_from_splice}
\end{figure}
\end{Example}

\begin{Example} Consider the manifold $Y$ given by the following plumbing:
$$
 \includegraphics[scale=1.1]{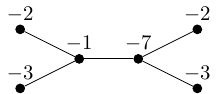}
$$
Then $Y$ is an integral homology sphere with $\Delta(Y) = \frac{1}{2}$ and the minimizing vectors are $\{\pm v, \pm w \}$, where $v=(1, 1, -1, -1, 1, 1)$, $w = (1, 1, -1, 1, -1, 1)$. Again, \smash{$\Delta(Y) < - \frac{\gamma(Y)}{4} + \frac{1}{2} = \frac{5}{2}$}. Note that this plumbing corresponds to the following splice diagram:
$$
 \includegraphics[scale=1.1]{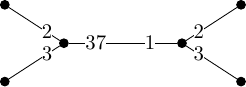}
$$
\end{Example}

\section{Upper bounds and cancellations}

In this section, we focus on upper bounds for $\Delta_b$. In principle, bounding $\Delta_b$ from above should be easy~-- the (shifted) norm $\bigl(-3s - \Tr(M) -l^2\bigr)/4$ of any element $l \in \mathcal{C}_b$ gives an upper bound.

However, when finding elements of $\LC_b$, we are facing two issues. First of all, $\LCt_b$ can be empty, e.g., for some $\spc$ structures on lens spaces, giving $\Zhat_b(q)=0$ and $\Delta_b = \infty$. Secondly, even if~$\LCt_b$ is non-empty, there may be drastic cancellations of the coefficients, as we will illustrate on the Seifert manifold in Example~\ref{ex:cancel_seifert}.

This prevents us from establishing general results in this direction. In particular, the vector $2u-\delta$ does not always give an upper bound for $\Delta_{\can}$, unlike for Seifert manifolds (where it was optimal), see Example~\ref{ex:cancel_H}.

We believe that those cancellations are rather special, being related to some additional symmetry of the plumbing graph. In particular, it would be rather surprising if they occurred for all $\spc$ structures at once. Therefore, we expect the following conjecture.

\begin{Conjecture}
For a negative-definite plumbed manifold $Y$, we have
 \[\min_{b \in \spc{Y}} \Delta_b \leq -\frac{\gamma(Y)}{4}+\frac12.\]
\end{Conjecture}

We now proceed with several examples of manifolds for which the cancellations occur. We~also compare \smash{$\Zhat(q)$} with the two-variable extension \smash{$\Zhathat$} defined in~\cite{AJK23}. The new variable $t$ can distinguish vectors $l$ with the same value of $l^2$, removing some cancellations, but not all of them, see Example~\ref{ex:t_cancel_0}.

\begin{Example}\label{ex:cancel_H}
 Consider the following plumbing:
$$
 \includegraphics[scale=1.1]{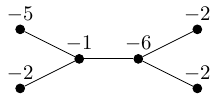}
$$
The resulting plumbed manifold $Y$ admits $20$ $\spc$ structures. For the canonical $\spc$ structure $\can$ we have \smash{\raisebox{0.3pt}{$\Delta_{\can} = \frac{33}{20}$}} with the (sole) minimizing vector $(1, -1, 1, -1, -1, 1)$. This is strictly larger than \smash{\raisebox{-0.3pt}{$-\frac{\gamma(Y)}{4} + \frac12 = \frac{13}{20}$}}.

The only vectors within $\LCt_{\can}$ that satisfy \smash{$\frac{1}{4}\bigl(-3s - \operatorname{Tr}(M) - l^2\bigr) = \frac{13}{20}$} are $l_1\! =\! (1, 1, -1, 1, 1, -1)$ and $l_2 = (-1, -1, 1, 1, 1, -3)$. However, their coefficients are \smash{$\ct_{l_1} = \frac{1}{4} = -\ct_{l_2}$} and so they cancel~out in \smash{$\Zhat_{\can}(q)$}.

The above cancellation does not happen for \smash{$\Zhathat_{\can}(q,t)$},
\[
\Zhathat_{\can}(q,t) = -\frac{1}{4}\bigl(\bigl(t^{-1} -t\bigr)q^\frac{13}{20} +q^\frac{33}{20} -t^{-1}q^\frac{53}{20} + \bigl(t^{-2} + t^2\bigr)q^\frac{73}{20} - t^{-1}q^\frac{93}{20} +\cdots\bigr).
\]
\end{Example}

\begin{Example}\label{ex:cancel_seifert}
 Consider Seifert manifold $M(2;(3,1),(3,2),(3,2))$. It is described by the following negative definite plumbing:
$$
 \includegraphics[scale=1.1]{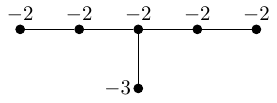}
$$
 It has $|H|$ = 9 and Casson--Walker invariant $\lambda = -2$. For the canonical $\spc$ structure $\can$ represented by $2u - \delta = ( 1, 1, 1, 0, 0, -1)$, the $q$-series $\Zhat_{\can}(q)$ is a single monomial
 \[
 \Zhat_{\can}(q) = \frac12 q^{-\frac{5}{6}}.
 \]
 This can be seen using the reduction theorem from~\cite[Theorem~4.2]{GKS23}. It implies that \smash{$\Zhat_{\can}(q)$} can be written as
 \begin{gather}
 q^{-\frac{5}{6}} \mathcal{L}_{27\cdot 9} \biggl(\se \frac{-1}{\bigl(z^{54}-1\bigr)}\biggr) \nonumber\\
 \qquad{}=\frac12
 q^{-\frac{5}{6}} \mathcal{L}_{27\cdot 9} \bigl( \bigl(
 1+z^{54}+z^{108}+z^{162} + \cdots\bigr) +
 \bigl(-z^{-54}-z^{-108}-z^{-162} - \cdots\bigr)\bigr)\nonumber\\
 \qquad{}= \frac12 q^{-\frac{5}{6}} \mathcal{L}_{27\cdot 9} (1) = \frac12 q^{-\frac{5}{6}}.\label{eq:reduction}
 \end{gather}

 Here the shortcut $\se$ denotes the symmetric expansion (average of expansions at $z=0$ and infinity), and the operator \smash{$\mathcal{L}_{27\cdot 9}$} is defined as
 \[
 \mathcal{L}_{27\cdot 9} (z^n) := q^{\frac{n^2}{4\cdot 27\cdot 9}}
 \]
  on monomials and extended linearly.
 The exponent \smash{$-\frac{5}{6}$} is computed as
 \[
 \frac{1}{4|H|}\Biggl(24\lambda -a_1 a_2 a_3\Biggl(-1+\sum_{i=1}^3 \frac{1}{a_i^2}\Biggr)\Biggr) = -\frac{5}{6}.
 \]
 In the last line of the equation \eqref{eq:reduction}, we used that $\mathcal{L}(z^n) = \mathcal{L}(z^{-n})$ for any $n \in \Z$. One could trace back this cancellation along the proof of the reduction theorem to find that all vectors except $2u - \delta$ \big(corresponding to the $z^0=1$ above\big) in the set $\LCt$ can be split into pairs~$l_1$,~$l_2$ satisfying $l_1^2=l_2^2$ and $\tilde{c}_{l_1} = -\tilde{c}_{l_2}$.
 Similarly to the previous example, \smash{$\Zhathat_{\can}(q,t)$} removes the cancellations and it gives an infinite series
 \[
 \Zhathat_{\can}(q,t) = \frac12 \bigl( q^{-\frac{5}{6}} t + q^{\frac{13}{6}} \bigl(1-t^{2}\bigr) + q^{\frac{67}{6}} \bigl(t^{-1}-t^{3}\bigr) + q^{\frac{157}{6}} \bigl(t^{-2} - t^{4}\bigr)+ \cdots\bigr).
 \]
 Note that this can be obtained by a reduction theorem for the \smash{$\Zhathat$} series, which was recently proved in~\cite{Liles2024}.
\end{Example}

\begin{Example}\label{ex:t_cancel_0}
We denote by $\Delta_{\can}'$ the smallest $q$-exponent in \smash{$\Zhathat_{\can}(q,t)$}. Consider the following plumbing:
$$
 \includegraphics[scale=1.1]{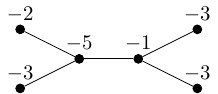}
$$
The resulting plumbed manifold $Y$ admits $21$ $\spc$ structures. For the canonical $\spc$ structure $\can$ represented by $2u - \delta = ( 1, 1, -1, 1, 1, -1)$, we have that \smash{$\Delta_{\can}' = \frac{5}{2}$} and the minimizing vectors for the quadratic form that produces $\Delta'_{\can}$ are given by $l_1 = (-1, -1, -3, 1, 1, 1)$, $l_2 = (-1, -1, 7, -1, -1, -1)$, $l_3 = (1, 1, -7, 1, 1, 1)$ and $l_4 = (1, 1, 3, -1, -1, -1)$ with coefficients \smash{$\ct_{l_1} = \ct_{l_3}= -\frac{1}{4}t^{-1}$} and \smash{$\ct_{l_2} = \ct_{l_4 }= -\frac{1}{4}t$}. This is strictly larger than \smash{$-\frac{\gamma(Y)}{4} + \frac12 = \frac{1}{2}$}.

For this manifold, the only vectors within $\tilde{\mathcal{C}}_{\can}$ which satisfy \smash{$\frac{1}{4}\bigl(-3s - \operatorname{Tr}(M) - l^2\bigr) = \frac{1}{2}$} are $l_1 = (-1, -1, 1, -1, -1, 1)$, $l_2 = (-1, -1, 3, 1, 1, -1)$, $l_3 = (1, 1, -1, 1, 1, -1)$ and $l_4 = (1, 1, -3,\allowbreak -1, -1, 1)$. However, their coefficients are \smash{$\ct_{l_1} = \frac{1}{4}t^{-1} = -\ct_{l_4}$} and \smash{$\ct_{l_2} = -\frac{1}{4}t = -\ct_{l_3}$} and so they cancel out in \smash{$\Zhathat_{\can}(q,t)$}. The entire $(q,t)$-series is given by
\[
\Zhathat_{\can}(q,t) = -\frac{1}{4}\bigl(\bigl(2t^{-1} + 2t\bigr)q^\frac{5}{2} + \bigl(2t^{-3} +2t^3\bigr)q^\frac{9}{2} + \bigl(-t^{-4} + t^{-2} +t^2 -t^4\bigr)q^\frac{15}{2} + \cdots\bigr).
\]

This example shows that cancellations do occur even for \smash{$\Zhathat_{\can}(q,t)$} and as a result
\[
\Delta_{\can}' > -\frac{\gamma(Y)}{4} + \frac{1}{2}.
\]
\end{Example}

\section{Comparison with correction terms}

In the last section, we compare $\Delta_b(Y)$ with correction terms $d_b(Y)=d(Y,[b])$ in Heegaard--Floer homology.\footnote{Again, we omit the brackets for $\spc$ structures.} We include this discussion because there were some expectations that $\Delta_b(Y)$ and~$d_b(Y)$ might be related~\cite{AJK23, GPP21}. In the Seifert case, we have an explicit formula for $\Delta_{\can}(Y)$ in terms of the $\gamma(Y)$ invariant by Theorem~\ref{thm:seifert}. We can then use elementary bounds for Dedekind sums to obtain estimates on $\Delta_{\can}(Y)$. Finally, we compare $\Delta_{\can}(Y)$ to $d_{\can}(Y)$ for some classes of Brieskorn spheres, where $d_{\can}(Y)$ is known, finding that they are very different.

\subsection{Quadratic forms}

Correction terms can be expressed as minimizers of a quadratic form over the characteristic vectors.

\begin{Theorem}[{\cite{Nem05}}]\label{thm:cor_terms} For an almost rational graph $\Gamma$, the correction terms are given by
 \[
 d_k(Y) = \max_{k' \in [k]} \frac{(k')^2+s}{4}.
 \]
\end{Theorem}
Note that this formula holds in many other cases and was conjectured by the second author to be true for any negative definite rational homology sphere~\cite{Nemethi_lattice08}.

To compare with $\Delta$, we need to shift to our conventions on $\spc$ structures, so we set $l = k' - Mu$. Then we have
 \begin{align*}
 \frac{(k')^2+s}{4} &= \frac{(l + Mu)^T M^{-1}(l + Mu) + s}{4}\\
 &= \biggl(\frac{\Tr(M) +3s}{4} + \frac{l^2}{4} \biggr) + \frac{l^T u}{2} - \frac12.
 \end{align*}
We see that the minimized quadratic forms for $d$ and $\Delta$ differ in the linear term
\[
\frac{l^T u}{2}= \frac{\bigl(\sum_{i=1}^s l_i\bigr)}{2}.
\]

\begin{Remark}
 In~\cite{GPP21}, the authors observed that in the setting of Theorem~\ref{thm:cor_terms}, the difference between $d_b(Y)$ and \smash{$\frac12-\Delta_b(Y)$} is an integer. For Seifert manifold $Y$ with the canonical $\spc$ structure, this can be explained as follows:
 By~\cite[Section~11.7.2]{Nem22},
 $d_{\can}(Y)$ and $\gamma(Y)$ (and therefore $\Delta_{\can}(Y)$) satisfy the following relation:
 \begin{equation*}
 d_{\can}(Y)=\frac{\gamma(Y)}{4} - 2\cdot \min\chi_{\can} \biggl( = -\Delta_{\can} + \frac12 - 2 \min \chi_{\can} \biggr).
 \end{equation*}
 Here $x\mapsto \chi_{\can}(x)$ is the holomorphic Euler characteristic of a sheaf $\mathcal{O}_x$ on the corresponding weighted homogeneous singularity with link $Y$, and $\min\chi_{\can}$ is its minimal value. In particular,
 \smash{$d_{\can}+\Delta_{\can}-\frac12 =-2\cdot \min\chi_{\can} $} is an even integer.
\end{Remark}

\subsubsection{Lower bound for Seifert manifolds}

We end this section with some elementary estimates considering the $\gamma(Y)$ invariant of Seifert manifolds, and hence of $\Delta_{\can}(Y)$. Theorem~\ref{thm:seifert} and the formula \eqref{eq:formula_gamma} for $\gamma(Y)$ give a formula for $\Delta_{\can}(Y)$,
 \begin{equation}\label{eq:formula_delta}
 \Delta_{\can} = - \frac{1}{4{\rm e}}\Biggl(2-n + \sum_{i=1}^n \frac{1}{a_i}\Biggr)^2 - \frac{{\rm e}+3}{4} - 3 \sum_{i=1}^n \ds(\omega_i, a_i).
 \end{equation}
We have the following well-known bound for Dedekind sums, for any $p>0$, $a \in \ZZ$:
\begin{equation*}
 -\mathbf{s}(1, p) \leq \mathbf{s}(a, p) \leq \mathbf{s}(1, p) = \frac{p}{12} + \frac{1}{6p} - \frac{1}{4}.
\end{equation*}
This estimate then gives the following.

\begin{Proposition}\label{lower-bound-delta-general-seifert}
 For the Seifert manifold $Y = M(b_0;(a_1, \omega_1), \dots, (a_n, \omega_n))$ and the canonical $\spc$ structure, we have
 \[
 \Delta_{\can}(Y) \leq - \frac{1}{4{\rm e}}\Biggl(2-n + \sum_{i=1}^n \frac{1}{a_i}\Biggr)^2 - \frac{{\rm e}+3(n+1)}{4} + \sum_{i=1}^n \biggl( \frac{a_i}{4} + \frac{1}{2a_i}\biggr)
 \]
 and
 \[
 \Delta_{\can}(Y) \geq -\frac{1}{4{\rm e}}\Biggl(2-n + \sum_{i=1}^n \frac{1}{a_i}\Biggr)^2 -\frac{{\rm e}+3(1-n)}{4} - \sum_{i=1}^n \biggl(\frac{a_i}{4} + \frac{1}{2a_i}\biggr),
 \]
 where ${\rm e}$ is the orbifold Euler number of $Y$.
\end{Proposition}

If $Y$ is a Brieskorn sphere $\Sigma(a_1, \dots, a_n)$, then $e = -\prod_{i=1}^n a_i^{-1}=-A^{-1}$. In this case, one can~write
\begin{equation*}
 \Delta(Y) = \frac{(n-2)^2A}{4} + (2-n) \sum_{i=1}^n \frac{A}{2a_i} + \sum_{i < j}\frac{A}{2a_ia_j} + \sum_{i=1}^n \frac{A}{4a_i^2} + \text{error terms},
\end{equation*}
where the error terms are bounded above by a linear function of the $a_i$. In particular, we see, up to these error terms, that $\Delta$ is a polynomial in $a_1, \dots, a_n$ of degree $n$ with leading term $(n-2)^2A/4$.

\subsubsection{Brieskorn spheres}

We illustrate the difference between $\Delta = \Delta_{\can}$ and $d = d_{\can}$ for some families of Brieskorn spheres, for which correction terms are explicitly known~\cite{BorNem2011}. For $p, q > 0$, set $\rho = (p-1)(q-1)/2$. Then for $\Sigma(p, q, p q+1)= S^3_{-1}(T_{p, q})$ we have
$\gamma(Y) = -4\rho(\rho-1)$. From Theorem~\ref{thm:seifert}, we immediately obtain the next result.

\begin{Corollary}
 For $Y = \Sigma(p, q, pq+1)$,
 \begin{equation*}
 \Delta(Y) = \rho(\rho-1) + \frac{1}{2}.
 \end{equation*}
\end{Corollary}
In contrast to this, the correction term vanishes, so that $\Delta(Y) \gg d(Y) = 0$.

For $Y = \Sigma(p, p+1, p(p+1)-1)$, we have
\begin{equation*}
d(Y) = \biggl\lfloor \frac{p}{2} \biggr\rfloor\biggl(\biggl\lfloor \frac{p}{2} \biggr\rfloor+1\biggr).
\end{equation*}
From Proposition~\ref{lower-bound-delta-general-seifert}, we obtain, after some manipulations, that for $p\geq 4$
\begin{equation*}
 \Delta(Y) \geq \frac{1}{4}\bigl(p^3-9p- 3 \bigr) > d(Y).
\end{equation*}

The case of $p = 2$ corresponds to the Poincar\'e sphere $Y = \Sigma(2, 3, 5)$, where we have $\Delta(Y) = \smash{-\frac{3}{2}} < d(Y) = 2$. The case $p = 3$ corresponds to the Brieskorn sphere $Y = \Sigma(3, 4, 11)$, where we have \smash{$\Delta(Y) = \frac{1}{2} < d(Y) = 2$}. In fact, for the family $Y_r = \Sigma(2, 3, 6r-1)$, it was computed in~\cite{GPP21} that \smash{$\Delta(Y_r) = -\frac{3}{2} < d(Y_r) = 2$} for all $r$. This seems to be a boundary case: the orbifold Euler number ${\rm e}$ is a linear function of $r$ and the contribution of the Dedekind sums in~\eqref{eq:formula_delta} is significant (cf.~the discussion after Proposition~\ref{lower-bound-delta-general-seifert}). We expect that for sufficiently general Seifert manifolds, one has
$d_{\can} < \Delta_{\can}$.

\appendix

\section{Equivalence of negative and weakly negative plumbings}

\subsection{Statement} Let $\Gamma$ be a connected negative definite plumbing graph whose associated plumbed 3-manifold $Y(\Gamma)$ is a $\QHS$. This means that $\Gamma$ is a tree, all genus decorations are zero, and $\det(M(\Gamma))\not=0$, where $M(\Gamma)$ is the plumbing matrix of $\Gamma$
 defined as in Section~\ref{ss:plumbing}. We denote the $\Q$-vector space generated by the vertices of $\Gamma$ by
 $\Q^s$.

 Let $M^{-1}$ be the inverse of $M$ (over the rational numbers), and let \smash{$M^{-1}|_{\Q^{s_3}}$}
 be its restriction to the subvector space of $\Q^s$ generated by vertices of degree $\delta_i\geq 3$.

We say that (cf.~\cite[Definition~4.3]{GM21}) $\Gamma$ is a \emph{weakly negative definite} plumbing graph
if ${M\!=\!M(\Gamma)}$ is nondegenerate and \smash{$M^{-1}|_{\Q^{s_3}}$} is a negative definite form.
Recall that $\Gamma$ is a {\it negative definite }
if~${M=M(\Gamma)}$ \big(or, equivalently, $M^{-1}$\big) is negative definite.

The main result of this appendix is the following.

\begin{Theorem}\label{th:app}
Let $\Gamma$ be a connected weakly negative definite plumbing graph. Then $\Gamma$ can be modified
 by plumbing calculus $($Neumann moves$)$ into another connected plumbing graph $\Gamma'$ such that
$\Gamma'$ is negative definite.

In particular, the family of $3$-manifolds associated with connected weakly negative definite plumbing graphs, respectively connected negative definite plumbing graphs agree.
\end{Theorem}

This theorem implies that the properties of any invariant associated with connected weakly negative definite plumbing graphs which is stable under the plumbing calculus of (oriented) connected plumbing graphs can be read from an equivalent
connected negative definite graph as well.

In the proof, we will use the following terminology.
Let $\Gamma$ be a plumbing graph, and let $\mathcal{N}$ be the set of \emph{nodes}, the vertices with valency $\delta_i\geq 3$. We call the connected components of the
full subgraph $\Gamma\setminus \mathcal{N}$
generated by vertices with valency $\delta_i\leq 2$ the \emph{strings of $\Gamma$}.

The proof of Theorem \ref{th:app} has two main steps.

\subsection{Step 1: Plumbing calculus}
Along the first step, we replace the connected weakly negative definite graph $\Gamma$ via plumbing calculus
by a new graph $\Gamma'$, which is connected
weakly negative definite, all the decorations of the strings are $\leq -2$, and we do not create new nodes along these moves. Along this procedure we use only such moves which preserve the connectedness of the graph and the weakly negative definiteness of the form,
and which ultimately replaces the strings into negative definite strings, and also might eliminate some
nodes but do not create new nodes.

We emphasize these details regarding the sequence of moves, since it is not true that any Neumann move preserves the weakly negative property, see Example~\ref{ex:app1}.
That is, e.g., if we increase the number of nodes by a Neumann move, even if the graph with smaller nodes is weakly negative definite, the larger graph might not be.

In particular,~\cite[Remark~4.7]{GM21}
is not correct in the generality how it was stated. However, it is correct for those moves which do not increase the number of nodes (see the proof below).

More details and proofs regarding this first step will be given in Appendix~\ref{ss:8.4}.
\begin{Example}\label{ex:app1}
Consider the following two moves:
$$
 \includegraphics[scale=1]{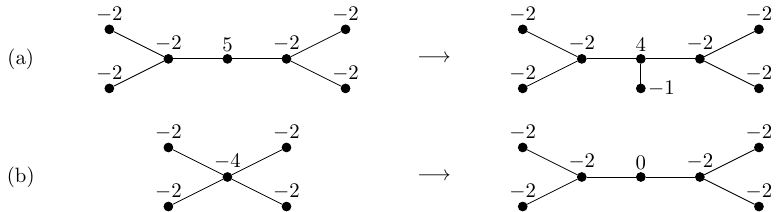}
$$

The graphs on the left hand side are connected weakly negative definite, while the graphs on the right hand side are \emph{not} weakly negative definite. The same holds for the positive blow up of the first graph.
\end{Example}

\subsection[Step 2: Negative definiteness of Gamma']{Step 2: Negative definiteness of $\boldsymbol{\Gamma'}$}
As a second step, we prove the following proposition.

\begin{Proposition}\label{prop:app}
Assume that $\Gamma$ is a weakly negative definite plumbing graph such that all its strings are negative definite $($for this it is enough that all the string decorations are $\leq -2)$.
 Then~$\Gamma$ itself is negative definite.
\end{Proposition}

In fact, a more general statement can be proved. Assume that we partition the set of vertices into two disjoint subsets: $V=V'\sqcup V''$ of cardinalities $s'$ and $s''$ respectively, and accordingly we have the direct sum decomposition \smash{$\Q^s=\Q^{s'}\oplus \Q^{s''}$}.
Assume that the matrix $M$ is nondegenerate and let us decompose $M$ and $M^{-1}$ into blocks according to the decomposition \smash{$\Q^s=\Q^{s'}\oplus \Q^{s''}$} as follows:
\[
M=\begin{pmatrix} A& B\\ C& D\end{pmatrix}, \qquad
M^{-1}=\begin{pmatrix} A'& B'\\ C'& D'\end{pmatrix}.
\]
\begin{Proposition}\label{prop:app2}
 Using the above notation, assume that $M$ is nondegenerate $($and symmetric$)$. Then if
 $D$ and $A'$ are negative definite then $M$ itself is negative definite.
\end{Proposition}

\begin{proof}
 Since $D$ is nondegenerate, we can write
\[
\begin{pmatrix} A& B\\ C& D\end{pmatrix}=
\begin{pmatrix} I& BD^{-1}\\ 0& I\end{pmatrix}
\begin{pmatrix} S& 0\\ 0& D\end{pmatrix}
\begin{pmatrix} I& 0\\ D^{-1}C& I\end{pmatrix},
\]
where $S:= A-BD^{-1}C$. Since $M$ is nondegenerate, $S$ also must be nondegenerate.
Inverting~the above decomposition, we get that in fact $S^{-1}=A'$. Therefore, by the hypothesis regarding~$A'$,~$S$ is negative definite too. That is, the matrix
\smash{$\bigl(\begin{smallmatrix} S& 0\\ 0& D\end{smallmatrix}\bigr)$} is negative definite.
Then by the symmetry of~the matrices, the transposed matrix \smash{$\bigl(BD^{-1}\bigr)^T$} equals
$D^{-1}C$, hence $M$ itself is negative definite.
\end{proof}

The matrix $S$ is called the \emph{Schur complement}. In the proof we used as model the computations from~\cite[p.~526]{Nem22}, where some other applications regarding negative definite graphs can be found.

 \begin{proof}[Proof of Proposition \ref{prop:app}]
Apply Proposition \ref{prop:app2} when $V'$ (resp.\ $V''$) is the set of
 nodes (resp.\ vertices of the strings).
\end{proof}

\subsection{Proof of step 1}\label{ss:8.4}
Let us return back to the proof of the first step.

In this step, we modify the {\it strings}
of the graph as follows: we blow up edges creating new~$-1$ vertices,
we blow down the $\pm 1$ vertices, and we also perform 0-chain absorptions
(similarly to
how in~\cite[Theorem~4.1]{Neu81} the `normal form' of graphs is obtained).
In this way, we either transform the string to a negative definite
string (with all its entries $\leq -2$), to an edge connecting two nodes, or we eliminate it completely
(decreasing the number of nodes).

Before we start the discussion of these cases, let us mention that along
\begin{itemize}\itemsep=0pt
 \item a $(+1)$ blow up/down the determinant of the plumbing matrix $M$
stays stable;
 \item a $(-1)$ blow up/down and 0-chain absorption $\det(M)$ changes its sign.
\end{itemize}

Recall also that \smash{$M^{-1}_{uv}=(-1)^{u+v} \det\bigl(M^{\rm adj}_{uv}\bigr)/\det(M)$},
where $M^{\rm adj}_{uv}$ is obtained from $M$ by deleting the $u$-th row and $v$-th column. In the case of matrices associated with plumbing trees,
\smash{$\det\bigl(M^{\rm adj}_{uv}\bigr)$} can be interpreted as determinants of certain subgraphs of $\Gamma$. Indeed,
fix $u,v\in V$ and let $p_{uv}$ be the shortest path in $\Gamma$ connecting $u$ and $v$. Let $\Gamma\setminus p_{uv}$ be the full subgraph obtained from~$\Gamma$ by deleting the vertices of $p_{uv}$ and adjacent edges.
Then
\begin{equation}\label{eq:app1}
M^{-1}_{uv}=-\frac{\det(-M(\Gamma\setminus p_{uv}))}{\det(-M)}.
\end{equation}

Assume that we make a move that preserves the number of nodes. Then the subspace generated by the nodes stays stable too, hence the entries $M^{-1}_{uv}$ before and after the move can be naturally compared. Furthermore, using the above facts regarding determinants and also~\eqref{eq:app1} we obtain that $M^{-1}_{uu}$ stays unmodified, and
for $u\not=v$ the entry $M^{-1}_{uv}$ either stays stable or it changes its sign.

Therefore, if we change conveniently base elements by their opposite in the subspace of the nodes, we can identify the forms $M^{-1}|_{\Q^{|\mathcal{N}|}}$, hence they are negative definite simultaneously.

Finally, we analyze those moves that decrease the number of nodes. We have to discuss several cases.

(a) Assume that $u$ is a $(\pm 1)$-vertex of valency one, it is connected to a vertex $v$
with valency three. Then, if we blow down $u$ then the node $v$ becomes a non-node in the new graph. However, exactly as in the previous discussion (when we preserved the number of nodes) we obtain that~$M^{-1}$ restricted to the
subspace of nodes in the first graph can be identified (after a base change)
with~$M^{-1}$ restricted to the
subspace generated by the nodes and $v$ in the second graph.
Since the first one is negative definite, the second one is negative definite too,
hence on any subspace (i.e., on the subspace generated by the nodes of the second graph) is negative definite as well.

(b)
Assume that $u$ is a vertex decorated by 0, it has valency two, and both of its
adjacent vertices are nodes $v$, $w$ (with decorations $e$ and $f$).
Then after 0-chain absorption we eliminate~$u$ and its two edges, and identify
$v$ and $w$ replacing them by a single node with decoration $e+f$.
Then again, $M^{-1}$ restricted to the
subspace generated by the nodes in the new graph can be identified with
the form of the original situation restricted to a subspace of codimension one.
Hence again it is negative definite. (This step in the opposite direction is not working, see Example~\ref{ex:app1}.)

(c) We have to analyze one more case.
Assume that in $\Gamma$, we have a vertex $u$ of valency one with decoration $0$, connected to a vertex $v$
with valency at least three. In this case, the 0-chain absorption eliminates both $u$ and $v$, hence the resulting graph is not connected.

However, this situation can never occur with $\Gamma$ weakly negative definite.
Indeed, in this case, by \eqref{eq:app1}, $M^{-1}_{vv}$ is 0, hence
$\Gamma$ is not weakly negative definite.

\subsection*{Acknowledgements}

We are grateful to Sergei Gukov, Mrunmay Jagadale and Sunghyuk Park for useful discussions. We would also like to thank to the anonymous referees for their valuable suggestions. A.~N\'emethi was partially supported by `\'Elvonal (Frontier)' Grant KKP 144148. J.~Svoboda was supported by the Simons Foundation Grant {\it New structures in low-dimensional topology}. S.~Harichurn was supported by the 2020 FirstRand FNB Fund Education Scholarship Award and the University of KwaZulu-Natal's 2024 Vincent Maphai Scholarship Award.

\pdfbookmark[1]{References}{ref}
\LastPageEnding

\end{document}